\documentclass[11pt]{article}
\usepackage{color,amsfonts,amssymb}
\usepackage{amsfonts,epsf,amsmath}
\usepackage{amsthm}
\usepackage{latexsym}
\usepackage{graphicx}
\usepackage{array,float}
\usepackage[margin=2cm]{geometry}
\usepackage{enumitem}
\usepackage[mathscr]{euscript}
\usepackage{array,float}

\usepackage{graphicx,latexsym}
\usepackage{lscape}
\usepackage{multicol}
\usepackage{multirow}
\usepackage{mathrsfs}
\usepackage{graphicx}
\usepackage{lipsum}
\usepackage{subfig}
\usepackage{graphicx}
\usepackage{subfig}
\usepackage{autobreak}
\allowdisplaybreaks
\usepackage{cite}
\usepackage{filecontents}
\usepackage{hyperref}

\newtheorem{thm}{Theorem}

\newtheorem{lem}{Lemma}

\newtheorem{remark}{Remark}

\newtheorem{discu}{Discussion:}

\newtheorem{conje}{Conjecture:}

\usepackage{fancyhdr}
\usepackage{mathtools}
\begin{document}
	\title{\textbf{Edge Metric Dimension of Silicate Networks}}
\author{\begin{tabular}{rcl}
		\textbf{S. Prabhu$^{\text a, }$\thanks{Corresponding author: drsavariprabhu@gmail.com}, T. Jenifer Janany$^{\text a}$}
	\end{tabular}\\
	\begin{tabular}{ccc}
		\small$^{\text a}$Department of Mathematics, Rajalakshmi Engineering College, Chennai 602105, India \\
	\end{tabular}}
	\maketitle
	\vspace{-0.5 cm}
	\begin{abstract}
	\baselineskip16pt
		Metric dimension is an essential parameter in graph theory that aids in addressing issues pertaining to information retrieval, localization, network design, and chemistry through the identification of the least possible number of elements necessary to identify the distances between vertices in a graph uniquely. A variant of metric dimension, called the edge metric dimension focuses on distinguishing the edges in a graph $G$, with a vertex subset. The minimum possible number of vertices in such a set is denoted as $\dim_E(G)$. This paper presents the precise edge metric dimension of silicate networks.
	\end{abstract}
	
	
	
	%%%%%%%%%%%%%%%%%%%%%%%%%%
	\section{Introduction}
	%%%%%%%%%%%%%%%%%%%%%%%%%%
	Metric dimension is a measure of how efficiently one can locate and distinguish between the vertices (nodes) of a graph using a minimal set of landmarks or reference points \cite{Sl75, HaMe76}. Metric dimension has applications in diverse domains, particularly network design \cite{BeEbEr06}, chemistry \cite{Jo93, ChErJo00}, robotics \cite{KhRaRo96}, and location-based services. Calculating the exact metric dimension of a graph is often a computationally challenging problem. For particular class of graphs, such as trees, there are efficient algorithms to find the metric dimension. However, for general graphs \cite{KhRaRo96}, bipartite graphs \cite{MaAbRa08} and directed graphs \cite{RaRaCy14}, this problem remains NP-hard. Despite the computational difficulty, the precise value of metric dimension is evaluated for many graph structures including honeycomb \cite{MaRaRa08}, TiO$_2$ nanotubes \cite{PrFlAr18}, butterfly \cite{MaAbRa08}, benes \cite{MaAbRa08}, Sierpi\'{n}ski \cite{KlZe13}, silicate \cite{MaRa11}, and irregular triangular networks \cite{PrJeAr23}.
	
	The distance $d_G(u,v)$ between two vertices $u$ and $v$ in a connected graph $G$, where $V(G)$ is the set of vertices and $E(G)$ is the set of edges, is defined as the minimum number of edges in any shortest path (geodesic) connecting them. The distance between an edge $e= vw$ and a vertex $u$ is given by, $d_G(e,u) = \min \{d_G(v,u), d_G(w,u)\}$. Based on the distance from vertices of an ordered subset $X= \{u_1, u_2, \ldots, u_l\}$, every vertex $v$ is represented with a vector of distance, $$ r(v| X) = (d_G(v,u_1), d_G(v, u_2), \ldots, d_G(v, u_l)).$$ The subset $X$ is a resolving set (\textbf{RS}) if $r(x|X) \neq r(y|X)$, $\forall x,y \in V(G)$. To put it in another way, $X$ is a \textbf{RS} if for every $x,y \in V(G)$, a vertex $u \in X$ exists such that distances from $x$ and $y$ to $u$ are unequal. See Figure \ref{defn}(a). We use the term (metric) basis to represent a \textbf{RS} with minimum vertices and metric dimension (\textbf{MD}), $\dim (G)$ to indicate its cardinality.
	\begin{figure}[H] 
		\centering
		\subfloat[]{\includegraphics[scale=.4]{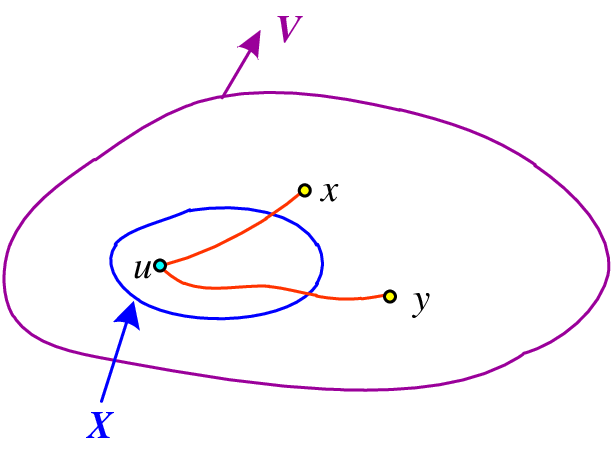}} 
		\quad 	\quad 	\quad 	\quad
		\subfloat[]{\includegraphics[scale=.4]{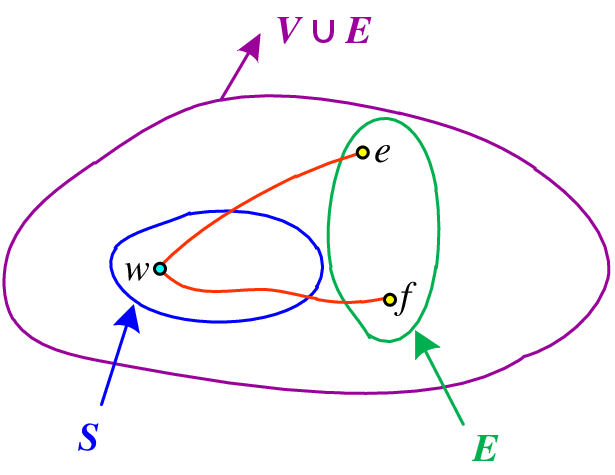}}
		\caption{(a) Resolving set (\textbf{RS}); (b) Edge resolving set (\textbf{ERS})} \label{defn}
	\end{figure}
	The vertices of basis distinguishes the vertices in the graph and not necessarily the edges. This fact ignited the spark of edge metric basis (\textbf{EMB}). A vertex subset $S$, is called an edge resolving set (\textbf{ERS}) when we have a vertex $w \in S$ for every pair of edges $e,f$, such that $d_G(w,e)$ and $d_G(w,f)$ are different. See Figure \ref{defn}(b). \textbf{EMB} is a vertex subset which resolves the edges of the graph with minimum possible vertices and the cardinality of \textbf{EMB} is the edge metric dimension (\textbf{EMD}), $\dim_E(G)$. This idea was conceptualize in \cite{KeTrYe18} by Kelenc et al. and is concluded as NP-complete. In the aftermath of this concept's origin a number of articles has outbroken in this field. To list a few, we have the characterization of graphs with maximum \textbf{EMD} \cite{Zu18,ZhTaSh19}, the \textbf{EMD} of grid graphs \cite{KeTrYe18}, web graph, prism related graph, convex polytope antiprism, convex polytope graph \cite{ZhGa20}, benzenoid tripod structure \cite{AhHuAz21}, honeycomb network, hexagonal network \cite{AbRaSi22}, Erdos-Renyi random graph \cite{Zu21}, and generalized Petersen graph $P(n, 3)$ \cite{WaWaZh22}. The research regarding this realm has exapnded to graph operations such as join, lexicographic, corona \cite{PeYe20}, and hierarchical products \cite{KlTa21}. Identifying graphs with \textbf{EMD} lesser than \textbf{MD} received specific attention \cite{KnSkYe22, KnMaTo21}.
	
	\section{Silicate Networks}
	Silicates are among the most abundant minerals on Earth and are the primary components of many materials, including rocks, minerals, and glasses. A silicon atom joined with equally spaced oxygen atoms when oxygen atoms are placed at the four corners of a tetrahedron produces a silicate tetrahedron. The basic building block of silicate networks is the silicate tetrahedron, $SiO_4$. See Figure \ref{silicate}. In a silicate network (or sheet), each oxygen atom is shared between adjacent tetrahedra, leading to the formation of a continuous network. The silicate network was perceived as an interconnection network, inspired by the molecular structure of $SiO_4$. Various properties, including metric dimension \cite{MaRa11}, topological descriptors \cite{HaIm14, LiWaWa17}, power domination \cite{StRaRy15}, and fault-tolerant metric dimension \cite{PrMaAr22} have been discussed.
	\begin{figure}[H] 
		\centering
		\includegraphics[scale=.5]{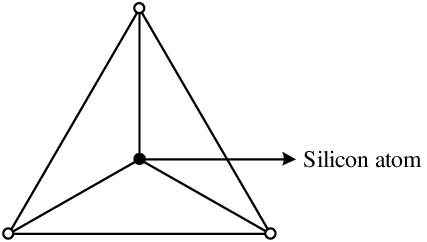}
		\caption {Silicate tetrahedron} 
		\label{silicate}    
	\end{figure}
	When a silicate tetrahedron connects with other tetrahedra linearly, a single-row silicate chain is created. A chain silicate with $n$ tetrahedra is denoted by $CS_n$. It has $3n+1$ vertices and $6n$ edges. See figure \ref{cs}.
	\begin{figure}[H] 
		\centering
		\includegraphics[scale=.6]{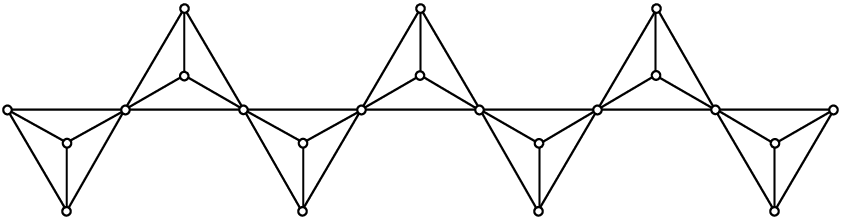}
		\caption {Chain silicate $CS_7$} 
		\label{cs}
	\end{figure}
	
	Cyclic silicates are a specific type of silicate mineral or compound characterized by a ring-like structure. In cyclic silicates, the silicate tetrahedra are arranged in closed loops or rings, giving rise to distinct structural and chemical properties. A cyclic silicate $CC_n$ can be obtained from a cycle of length $n$, by replacing each edge with a tetrahedron. $CC_n$ has $3n$ vertices and $6n$ edges. See Figure \ref{cc}.
	\begin{figure}[H] 
		\centering
		\includegraphics[scale=.6]{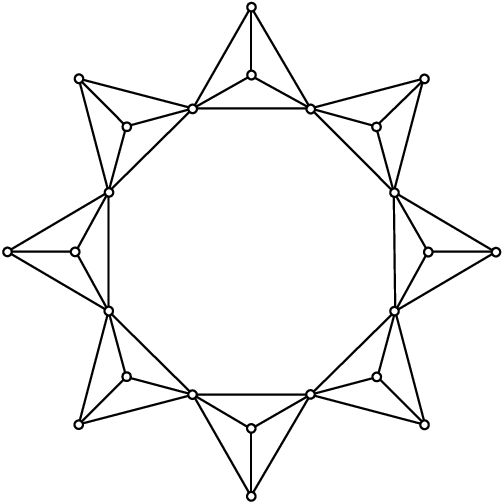}
		\caption {Cyclic silicate $CC_8$} 
		\label{cc}
	\end{figure}
	
	\section{Main Results}
	A one point union of two tetrahedrons is called a twin tetrahedron in silicate network. In $CS_n$ we have twin tetrahedrons with five 3-degree vertices and some with four 3-degree vertices. See Figure \ref{tt}.
	\begin{figure}[H] 
		\centering
		\subfloat[]{\includegraphics[scale=0.75]{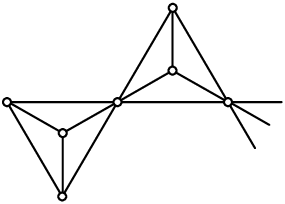}} 
		\quad \quad \quad   
		\subfloat[]{\includegraphics[scale=0.75]{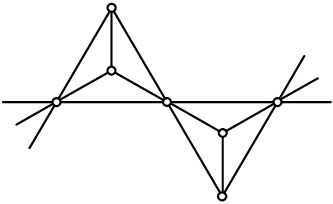}} 
		\caption {Twin tetrahedrons} \label{tt}
	\end{figure}
	Let us define a cubic set, $C$ to denote the set of 3-degree vertices in a twin tetrahedron. We have the following lemma regarding the role of cubic vertices in \textbf{ERS} of silicate networks.
	%\begin{lemma}
	%   Let $G$ be any silicate network and $S$ be any edge resolving set of $G$. Then $|C \setminus S| \leq 1$ for every twin tetrahedron i.e, for every twin tetraherdon at most one vertex of $C$ can be excluded from $S$.
	%\end{lemma}
	\begin{lem} \label{lbd}
		Let $G$ be any silicate network and $S$ be an \textbf{ERS} of $G$. Then for every twin tetraherdon at most one vertex of $C$ can be excluded from $S$ i.e, $|C \setminus S| \leq 1$.
	\end{lem}
	\begin{proof}
		Suppose there exists a twin tetrahedron with two cubic vertices $p,q \notin S$. Let $v_0$ be the intersection point of the twin tetrahedron and let $e_1= v_0p$, $e_2 = v_0q$. Then $d(e_1,u)=d(e_2,u)=d(v_0,u)$ for each $u \in V(G)\setminus \{p,q\}$, which contradicts $S$ being an \textbf{ERS}.
	\end{proof}
	Similar to the twin tetrahedrons, we have two types of tetrahedrons, some with three cubic vertices (Type I) and the remaining with two cubic vertices (Type II). See Figure \ref{k4}.
	\begin{figure}[H] 
		\centering
		\subfloat[]{\includegraphics[scale=1]{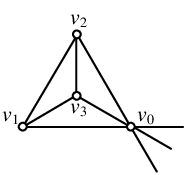}} 
		\quad \quad \quad \quad  
		\subfloat[]{\includegraphics[scale=1]{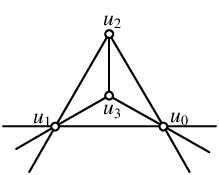}} 
		\caption {Tetrahedrons in silicate network} \label{k4}
	\end{figure}
	
	Consider a type I tetrahedron. Let $v_0, v_1, v_2, v_3$ be the vertices and the edges of the tetrahedron be $e_1= v_0v_1, e_2= v_0v_2, e_3= v_0v_3, e_4= v_1v_2, e_5= v_1v_3, e_6= v_2v_3$. Let $S_1= \{s, v_1, v_2\}$ and $d(s,v_0)=x$. Then
	\vspace{-0.5 cm}
	\begin{align*}
		r(e_1|S_1) &= (x,0,1)\\
		r(e_2|S_1) &= (x,1,0)\\
		r(e_3|S_1) &= (x,1,1)\\
		r(e_4|S_1) &= (x+1,0,0)\\
		r(e_5|S_1) &= (x+1,0,1)\\
		r(e_6|S_1) &= (x+1,1,0)
	\end{align*}
	Similarly, in type II tetrahedron let $u_0, u_1, u_2, u_3$ be the vertices and $e_1= u_0u_1, e_2= u_0u_2, e_3= u_0u_3, e_4= u_1u_2, e_5= u_1u_3, e_6= u_2u_3$ be the edges. If $S_2= \{s,t,u_2\}$, $d(s,u_0)=x$ and $d(t,u_1)=y$ then
	\vspace{-0.5 cm}
	\begin{align*}
		r(e_1|S_2) &= (x,y,1)\\
		r(e_2|S_2) &= (x,y+1,0)\\
		r(e_3|S_2) &= (x,y+1,1)\\
		r(e_4|S_2) &= (x+1,y,0)\\
		r(e_5|S_2) &= (x+1,y,1)\\
		r(e_6|S_2) &= (x+1,y+1,0)
	\end{align*}
	The following remark and lemma are straightforward from the above observations.
	\begin{remark} \label{k4lbd}
		For terahedrons of type I, $|S \cap C|\geq 2$ and for type II, $|S \cap C|\geq 1$, where $S$ is an \textbf{ERS} and $C$ is the set of 3-degree vertices.
	\end{remark}
	\begin{lem} \label{ubd}
		In a silicate network $G$, if for a vertex subset $S$, the following conditions holds, then $S$ is an \textbf{ERS}.
		\begin{enumerate}[label=(\roman*)]
			\item For every twin tetrahedron $|C\setminus S|\leq 1$
			\item For every type I tetrahedron, $|S \cap C|\geq 2$
			\item For every type II tetrahedron, $|S\cap C|\geq 1$
		\end{enumerate}
	\end{lem}
	%\begin{remark} \label{ubd}
	%If $|S \cap C|\geq 2$ for type I tetrahedron, $|S \cap C|\geq 1$ for type II tetrahedron and at least one cubic vertex of every tetrahedron belongs to $S$ then $S$ is an edge resolving set. 
	%\end{remark}
	\begin{thm}
		For any even integer $n \geq 2$, $\dim_E(CS_n)= \frac{3n}{2}+2$.
	\end{thm}
	\begin{proof}
		When $n$ is even, $CS_n$ can be partitioned into $\frac{n}{2}$ edge-disjoint twin tetrahedrons. Let $C_1, C_2, \ldots, C_{\frac{n}{2}}$ be the corresponding cubic sets. Since $|C_1|=|C_{\frac{n}{2}}|=5$ and $|C_2|=\ldots=|C_{\frac{n}{2}-1}|=4$, by Lemma \ref{lbd}, $\dim_E\geq 3(\frac{n}{2}-2)+ 4(2) = \frac{3n}{2}+2$.\\
		Consider $CS_n$ as a path on $n$ vertices $\{w_1, w_2, \ldots, w_n\}$. Define a mapping $l: \{w_1, w_2, \ldots, w_n\} \rightarrow \{1,2\}$ as follows, $$l(w_1)=l(w_2)=l(w_{n-1})=l(w_n)=2$$
		for $3\leq i \leq n-2$, \vspace{-0.5cm} $$l(w_i)= \begin{cases}
			1, & i \equiv 1\ {\rm mod}\ 2\\
			2, & i \equiv 0\ {\rm mod}\ 2
		\end{cases}$$
		Consider a collection of 3-degree vertices S, with the property that $|S \cap R_i| =l(w_i)$, where $R_i$ is the set of vertices of the tetrahedron corresponding to $w_i$. By Lemma \ref{ubd}, $S$ is an \textbf{ERS} of $CS_n$.
	\end{proof}
	\begin{thm}
		For any odd integer $n \geq 1$, $\dim_E(CS_n)= \frac{3(n+1)}{2}$.
	\end{thm}
	\begin{proof}
		$CS_n$ with $n$ odd, has two twin tetrahedrons with five cubic vertices, $\frac{n-1}{2}-2$ twin tetrahedrons with four cubic vertices and a single tetrahedron with two cubic vertices. By Lemma \ref{lbd}, $\dim_E(CS_n) \geq \frac{3(n+1)}{2}$.\\
		Define a mapping $l: \{w_1, w_2, \ldots, w_n\} \rightarrow \{1,2\}$ as follows, $$l(w_1)=l(w_2)=l(w_{n-1})=l(w_n)=2$$ and for $3\leq i \leq n-2$, $l(w_i)= \begin{cases}
			1, & i \equiv 1\ {\rm mod}\ 2\\
			2, & i \equiv 0\ {\rm mod}\ 2
		\end{cases}$. Here each $w_i$ represents a tetrahedron of $CS_n$. Let $S$ be a set of 3-degree vertices with $|S \cap R_i| =l(w_i)$, where $R_i$ is the set of vertices of the tetrahedron corresponding to $w_i$. Since $S$ satisfies the conditions in Lemma \ref{ubd}, and $|S|= \frac{3(n+1)}{2}$, it follows that $\dim_E \leq \frac{3(n+1)}{2}$.
	\end{proof}
	\begin{thm}
		For any even integer $n \geq 4$, $\dim_E(CC_n)= \frac{3n}{2}$.
	\end{thm}
	\begin{proof}
		For even values of $n$, $CC_n$ has $\frac{n}{2}$ twin tetrahedrons that are edge-disjoint and each cubic set has four vertices. Thus by Lemma \ref{lbd}, $\dim_E \geq \frac{3n}{2}$.\\
		Consider $CC_n$ as a cycle on $n$ vertices $\{w_1, w_2, \ldots, w_n\}$. Define a mapping $l: \{w_1, w_2, \ldots, w_n\} \rightarrow \{1,2\}$ as follows, for $1\leq i \leq n$, $l(w_i)= \begin{cases}
			1, & i \equiv 1\ {\rm mod}\ 2\\
			2, & i \equiv 0\ {\rm mod}\ 2
		\end{cases}$.
		Let $S$ be a set of 3-degree vertices with $|S \cap R_i| =l(w_i)$, where $R_i$ is the set of vertices of the tetrahedron corresponding to $w_i$. By Lemma \ref{ubd}, $S$ resolves the edges of $CC_n$.
	\end{proof}
	\begin{thm}
		For any odd integer $n \geq 3$, $\dim_E(CC_n)= \frac{3(n+1)}{2}-1$.
	\end{thm}
	\begin{proof}
		In $CC_n$, when $n$ is odd, we have $\frac{n-1}{2}$ twin tetrahedron with four cubic vertices and a single tetrahedron. Thus, $\dim_E(CC_n)\geq \frac{3(n+1)}{2}-1$.\\
		Consider $CC_n$ as a cycle on $n$ vertices $\{w_1, w_2, \ldots, w_n\}$. Define a mapping $l: \{w_1, w_2, \ldots, w_n\} \rightarrow \{1,2\}$ as follows, for $1\leq i \leq n-1$, $l(w_i)= \begin{cases}
			1, & i \equiv 1\ {\rm mod}\ 2\\
			2, & i \equiv 0\ {\rm mod}\ 2
		\end{cases}$, and $l(w_n)=2$. Consider a collection of 3-degree vertices S, with the property that $|S \cap R_i| =l(w_i)$, where $R_i$ is the set of vertices of the tetrahedron corresponding to $w_i$. By Lemma \ref{ubd}, $S$ is an edge resolving set of $CC_n$.
	\end{proof}
	
	\section{Conclusion}
	Silicate networks are prevalent in nature, forming the backbone of many minerals in the Earth's crust. They also have significant industrial applications, particularly in the production of glasses, ceramics, and construction materials. Understanding the structure and properties of silicate networks is essential for designing and engineering materials with specific functionalities and characteristics. This study has identified the necessary vertices to resolve the edges in cyclic and chain silicate networks. Future works can focus on this parameter for other types of silicate networks.
	
	\section*{Acknowledgements}
	The authors are thankful to Rajalakshmi Educational Trust for granting financial support under Centre for Sponsored Research and Consultancy/CSRC84.
	\bibliographystyle{amsplain}
	\nocite{}
	\bibliography{biblio}
	
\end{document}